\documentclass[12pt]{article}
\usepackage{amsmath, amssymb, amsthm}
\usepackage{mathtools}
\usepackage{hyperref}
\usepackage{xcolor}
\usepackage{booktabs}
\usepackage{geometry}
\usepackage{tikz,tkz-graph}
\usetikzlibrary{decorations.markings, arrows.meta}

\newtheorem{theorem}{Theorem}
\newtheorem{lemma}[theorem]{Lemma}
\newtheorem{proposition}[theorem]{Proposition}
\newtheorem{conjecture}[theorem]{Conjecture}
\newtheorem{question}[theorem]{Question}
\newtheorem{corollary}[theorem]{Corollary}
\theoremstyle{definition}
\newtheorem{definition}[theorem]{Definition}
\newtheorem{example}[theorem]{Example}
\theoremstyle{remark}
\newtheorem*{remark}{Remark}

\newenvironment{salign}{
    \begin{equation}
    \begin{aligned}
}{
    \end{aligned}
    \end{equation}
    \ignorespacesafterend
}

\definecolor{changelogblue}{RGB}{0,90,200}
\definecolor{changelogred}{RGB}{200,30,30}

\newcommand{\defeq}{\stackrel{\textnormal{def}}{=}}

\newcommand{\prim}{%
  \mathbin{%
    \tikz[baseline=-0.2ex, x=1.15ex, y=1.15ex, line width=0.10ex, line cap=round]{
      \draw (0.2,0.00) -- (0.2,1.00);
      \draw (0.8,0.00) -- (0.8,1.00);
      \draw (0.00,0.2) -- (1.00,0.2);
      \draw (0.00,0.8) -- (1.00,0.8);
    }%
  }%
}

\def \C {{\mathbb C}}

\def \R {{\mathbb R}}

\def \Z {{\mathbb Z}}

\def \bP {\textnormal{\textbf P}}
\def \bE {\textnormal{\textbf E}}

\def \cD {{\mathcal D}}

\numberwithin{equation}{section}

\title{On the Largest Strongly Connected Component of Randomly Oriented Divisor Graphs}

\author{
  Jihyung Kim\\
  \and
  Tristan Phillips\\
}
\date{}

\begin{document}
\maketitle

\begin{abstract}
    We introduce the study of \textit{randomly oriented divisor graphs}. For each $\rho \in [0,1]$, the randomly oriented divisor graph $\cD_\rho(N)$ is obtained from the divisor graph on $\{1, 2, \ldots, N\}$ by directing each edge according to divisibility and independently reversing the direction of each edge with probability $\rho$. We study the expected size of the largest strongly connected component, $\bE[\#\Phi_{}(\cD_\rho(N))]$. Our main result gives a lower bound for this quantity in terms of the distribution of values of the divisor function $\tau(n)$. As a consequence, we show that for any fixed $\rho \in (0,1)$, the largest strongly connected component has expected size asymptotic to $N$. To obtain explicit bounds, we prove an effective version of a theorem of Hardy and Ramanujan on the normal order of $\log \tau(n)$, which may be of independent interest.
\end{abstract}


\section{Introduction}

The interplay between the multiplicative structure of the integers and graph theory has been a source of rich problems in combinatorial number theory.
The \textit{divisor graphs} $G_N$, whose vertices are labeled $\{1,2,\dots, N\}$ with edges between vertices labeled $a$ and $b$ if $a|b$ or $b|a$, encode arithmetic information in a combinatorial framework. 
Divisor graphs have been studied in the context of number theory (e.g., \cite{Pom83, ES95, Sai98, MS20, McN21}) and combinatorics (e.g., \cite{CMS01, Fra03, AAA10}). For a general survey on divisor graphs see \cite{RD23}. 

Divisor graphs have also been studied in network science (e.g., \cite{SZ10, SBA15, RA20}), where it has been noted that they share certain structural properties with real networks, such as being \textit{scale free}, i.e., their degree distribution follows a power law (asymptotically). Classical random graph models such as the Erd\H{o}s--R\'enyi graphs are not scale-free. Many real networks are not only scale-free, but also directed. The \textit{randomly oriented divisor graphs} introduced in this article inherit both the rich arithmetic structure coming from the divisor graph (making them scale-free), and also randomness coming from how we choose the orientation.   

A natural directed structure on $G_N$ is given by the \textit{oriented divisor graph} $D_N$, which is obtained from $G_N$ by directing each edge so that if $a|b$ then there is an edge from $b$ to $a$. This orientation reflects the poset structure of divisibility. In this paper we consider all orientations on $G_N$, by introducing a probabilistic structure on divisor graphs.

For each $\rho\in[0,1]$ the \textit{randomly oriented divisor graph} $\cD_\rho(N)$ is obtained from the oriented divisor graph $D_N$ by independently reversing the direction of each edge  with probability $\rho$. When $\rho=0$ one recovers the oriented divisor graph, but when $\rho\not\in \{0,1\}$ the randomly oriented graph interpolates between the arithmetic structure of divisibility and the randomly oriented graph on the same edge set. 

Randomly oriented graphs have been studied in various ways, including percolation theory (e.g., \cite{McD81, Gri01, Lin11, Nar18}) and random walks (e.g., \cite{CP04,GL07,CGPS11}). In most of these cases the underlying graph has a simple combinatorial structure (e.g., a lattice or complete graph). In contrast, the divisor graph possesses a rich arithmetic structure. The randomly oriented divisor graph thus provides a new setting in which tools from probabilistic number theory can be used to study questions in random graph theory, and vice versa.

\subsection{Statement of results}

The main object of study in this article is the expected value of the size of the largest strongly connected component of $\cD_\rho(N)$, denoted $\bE[\#\Phi_{}(\cD_\rho(N))]$. 

For $m\in \Z_{>0}$, let $\tau(m)$ denote the number of divisors of $m$. For example, 
\begin{equation}
    \tau(12)=\#\{1,2,3,4,6,12\}=6.
\end{equation}

Our first result gives a general lower bound in terms of the distribution of the divisor function.

\begin{theorem}\label{thm:Largest_SCC_abstract_bound}
    If
    \begin{equation}
        \#\{n\leq N : \tau(n)-2\geq x\} \geq y,
    \end{equation}
    then
    \begin{equation}\label{eq:LSCC_bound}
        \bE[\#\Phi_{}(\cD_{\rho}(N))] \geq  y \left(1 - \rho(2\rho-\rho^2)^{x} - (1-\rho)(1-\rho^2)^{x}\right).
    \end{equation}
\end{theorem}

For example, when $\rho=1/2$, the bound \eqref{eq:LSCC_bound} becomes
\begin{equation}
    \bE[\#\Phi_{}(\cD_{1/2}(N))] \geq y (1-(3/4)^x).
\end{equation}

The strength of the bound \eqref{eq:LSCC_bound} depends on understanding how often $\tau(n)$ is large. A celebrated result of Hardy and Ramanujan \cite{HR17} shows that $\log(\tau(n))$ has normal order\footnote{See Definition \ref{def:normal_order} for a precise definition of normal order.} $\log(2)\log\log(n)$, i.e., for \textit{most} $n$, $\tau(n)\approx \log(n)^{\log(2)}$. Combining Theorem \ref{thm:Largest_SCC_abstract_bound} with this fact yields:

\begin{corollary}
    Let $\rho\in (0,1)$. The expected size of the largest strongly connected component of $\cD_{\rho}(N)$ is asymptotic to $N$, i.e., 
    \begin{equation}
        \lim_{N\to\infty} \frac{\bE[\#\Phi_{}(\cD_{\rho}(N))]}{N} = 1.
    \end{equation}
\end{corollary}

In order to prove explicit bounds, we prove the following quantitative version of the result of Hardy and Ramanujan, which may be of independent interest.

\begin{theorem}\label{thm:divisors_Durkan}
    Let $\epsilon\in (0,1)$, and define
    \begin{equation}\label{eq:c_epsilon_def}
        c_\epsilon \defeq \frac{4}{\epsilon}\left(\frac{e}{1-\epsilon}\right)^{4.096}.
    \end{equation}
    Then for every $N\geq 3$,
    \begin{equation}\label{eq:divisors_Durkan}
        \#\left\{ n\leq N : \tau(n) \geq \log(N)^{\log(2)(1-\epsilon)}\right\} \geq N - \frac{c_\epsilon\, N}{\log(N)^{\epsilon^2/2}}.
    \end{equation}
\end{theorem}

The proof of Theorem \ref{thm:divisors_Durkan} combines the elementary bound $\tau(n)\geq 2^{\omega(n)}$ with an upper bound of Durkan \cite{Dur23} for the level sets of $\omega(n)$, followed by an elementary lower-tail estimate. The exceptional term $c_\epsilon/\log(N)^{\epsilon^2/2}$ decays as a negative power of $\log(N)$; the bound \eqref{eq:divisors_Durkan} holds for all $N\geq 3$, but is nontrivial only once $\log(N)^{\epsilon^2/2}>c_\epsilon$, that is, once $N>\exp\!\big(c_\epsilon^{2/\epsilon^2}\big)$. We have made no attempt to optimise the constant $c_\epsilon$.

Combining Theorem \ref{thm:Largest_SCC_abstract_bound} with Theorem \ref{thm:divisors_Durkan} yields the following lower bound for the expected size of the largest strongly connected component when $N$ is sufficiently large:

\begin{corollary}\label{cor:Explicit_Largest_SCC_large_N}
    Let $\epsilon\in(0,1)$ and $\rho\in(0,1)$, and let $c_\epsilon$ be as in \eqref{eq:c_epsilon_def}. For each $N\geq 3$, define
    \begin{equation}
        x_N \defeq \log(N)^{\log(2)(1-\epsilon)}-2.
    \end{equation} Then 
    \begin{equation}
        \bE[\#\Phi_{}(\cD_{\rho}(N))] \geq  N \left( 1 - \frac{c_\epsilon}{\log(N)^{\epsilon^2/2}}\right) \left(1 - \rho(2\rho-\rho^2)^{x_N} - (1-\rho)(1-\rho^2)^{x_N}\right).
    \end{equation}
\end{corollary}

We also provide a complementary bound  (Corollary \ref{cor:Explicit_Largest_SCC_Primorial}) for arbitrary $N$, at the cost of weaker asymptotics.
This bound is a consequence of Theorem \ref{thm:Largest_SCC_abstract_bound} and Proposition \ref{prop:primorial_bound}. 

\begin{corollary}\label{cor:Explicit_Largest_SCC_Primorial}
    Let $N\geq 1$, $\rho\in(0,1)$, and $x\in \Z_{\geq 0}$. Let $d=\lceil\log_2(x+2)\rceil$, and let $p_d\prim$ denote the product of the first $d$ primes. Then 
    \begin{equation}
        \bE[\#\Phi_{}(\cD_{\rho}(N))] \geq  \left\lfloor \frac{N}{p_{d}\prim} \right\rfloor \left(1 - \rho(2\rho-\rho^2)^{x} - (1-\rho)(1-\rho^2)^{x}\right).
    \end{equation}
\end{corollary}

Finally, we investigate the diameter of randomly oriented divisor graphs. 
The diameter of the divisor graph $G_N$ is $2$ (whenever $N\geq 3$), since every vertex is connected to the vertex $1$, and the shortest path between $N$ and $N-1$ has length $2$. 
We conjecture that the expected size of the diameter of the largest strongly connected component of the randomly oriented divisor graph is much larger.

\begin{conjecture}\label{conj:diameter}
    The expected value of the diameter of the largest strongly connected component of $\cD_\rho(N)$ grows like $\log(N)$ asymptotically, i.e., there exists a constant $c_\rho$, depending only on $\rho$, such that
    \begin{equation}
        \lim_{N\to\infty} \frac{\bE\left[\textnormal{Diameter of $\cD_\rho(N)$}\right]}{\log(N)}=c_\rho.
    \end{equation}
\end{conjecture}
In Section \ref{sec:simulations} we discuss computational support for this conjecture.
Our code for running simulations for randomly oriented divisor graphs can be found in the \texttt{GitHub} repository \cite{KP26}.

\subsection{Organization}

In Section \ref{sec:setup} the precise definition of randomly oriented divisor graphs is given.
In Section \ref{sec:SCC} the largest strongly connected component of a randomly oriented graph is defined as a random variable. 
In Section \ref{sec:LSCC_Thm} Theorem \ref{thm:Largest_SCC_abstract_bound} is proven using basic facts about graphs and probability. 
In Section \ref{sec:divisors} Theorem \ref{thm:divisors_Durkan} and Proposition \ref{prop:primorial_bound} are proven. 
In Section \ref{sec:simulations} we provide computation data from simulations, for both the largest strongly connected component and the diameter. 
Finally, in Section \ref{sec:questions} we pose several open questions regarding randomly oriented divisor graphs.

\subsection*{Acknowledgments}
Jihyung Kim (JK) would like to thank Keeheon Lee for introducing him to network science, which motivated him to start this research project.
%
 %
%
Tristan Phillips (TP) is very thankful for JK's generosity in sharing with him his ideas for this project.
 We would also like to thank Aidan Hennessey, Andrew Kobin, and  Alex Moon for interesting conversations related to randomly oriented divisor graphs. 
Special thanks to Carl Pomerance for pointing out the reference \cite{Dur23} and sketching the proof of Theorem \ref{thm:divisors_Durkan}, which significantly improved upon the bound in our original version of the theorem.
JK was supported by a research scholarship from the William H.\@ Neukom Institute for Computational Science at Dartmouth College.
 TP was supported by the National Science Foundation, via grant DMS-2303011.

\section{Setup and definitions}\label{sec:setup}

In this section we give a precise definition of randomly oriented divisor graphs. 

\begin{definition}
    A \textbf{(simple) graph} $G$ is a pair $(V,E)$, where:    \begin{itemize}
        \item $V$ is a set of elements called \textbf{vertices}.
        \item $E$ is a set of two element subsets of $V$, called \textbf{edges}. 
    \end{itemize}
\end{definition}

\begin{definition}
    A \textbf{directed graph} $D$ is a pair $(V,E)$, where:    \begin{itemize}
        \item $V$ is a  set of elements called \textbf{vertices}.
        \item $E$ is a set of ordered pairs of distinct vertices, called \textbf{(directed) edges}. 
    \end{itemize}
\end{definition}

\begin{definition}
    An \textbf{oriented graph} is a directed graph $D=(V,E)$ with the property that  if $(v_1,v_2)\in E$ then $(v_2,v_1)\not\in E$. 
\end{definition}

\begin{definition}
    An \textbf{orientation} of a graph $G=(V,F)$ is an oriented graph $D=(V,E)$  with the property that if $(v_1,v_2)\in E$, then $\{v_1,v_2\}\in F$, and if $\{w_1,w_2\}\in F$, then $(w_1,w_2)$ or $(w_2,w_1)$ are in $E$.
\end{definition}

Let $G=(V,F)$ be a graph. Let $\Omega\defeq \Omega(G)$ denote the set of all orientations of the underlying graph $G$.
Let $D=(V,E)\in \Omega(G)$ be a fixed orientation on $G$. 

 Define a function
\begin{salign}
    f: \Omega &\to \Z_{\geq 0}\\
    D'=(V,E') &\mapsto \#\{(w,v)\in E' : (v,w)\in E\}.
\end{salign}
This function counts the number of edges whose orientations differ between $D$ and $D'$. 

\begin{definition}
For $\rho\in[0,1]$, the \textbf{randomly oriented graph} $\cD_\rho=(D,\rho)$ is a random variable defined on $\Omega$, determined by
\begin{equation}
    \bP(\cD_\rho = D') = \rho^{f(D')}(1-\rho)^{\# E - f(D')}.
\end{equation}
\end{definition}


   The \textbf{divisor graph} $G_N=(V_N,F_N)$ is the graph on $N$ vertices, labeled $1$ through $N$, with an edge between vertices labeled $a$ and $b$ if and only if $a|b$ or $b|a$. 

   The \textbf{oriented divisor graph} $D_N=(V_N,E_N)$ is the orientation on $G_N$, were the edges are oriented so that if $a>b$ and $b|a$, then $(a,b)\in E_N$. 

    The \textbf{randomly oriented divisor graph} $\cD_\rho(N)$ is the randomly oriented graph $(D_N,\rho)$.

\section{Strongly connected components}\label{sec:SCC}

In this section we define the largest strongly connected component of a randomly oriented graph as a random variable an the set of orientations of a fixed graph. We give some examples of the expected value of the largest strongly connected component of some randomly oriented divisor graphs.

\begin{definition}
    A directed graph is \textbf{strongly connected} if there is a directed path between each pair of its vertices.
\end{definition}

\begin{definition}
    A \textbf{strongly connected component} of a directed graph is a subgraph that is strongly connected and is not contained in a larger strongly connected subgraph.
\end{definition}

Strongly connected components give a partition of the vertices of a directed graph. For a finite oriented graph $D=(V,E)$, let $C(D)=\{(V_1,E_1),\dots,(V_n,E_n)\}$ be the set of subgraphs with the property that each $(V_i, E_i)$ is strongly connected and $V=\bigsqcup_{i\leq n} V_i$. We will refer to elements of $C(D)$ as \textbf{components} of $D$. 
For $v\in V$ let $(V_v, E_v)$ denote the component $(V_i, E_i)$ of $D$ for which $v\in V_i$. 

Define the map
\begin{salign}
    \Phi_v : \Omega(G) &\to \bigcup_{H\subseteq G} \Omega(H) \\
    D' &\mapsto (V_v, E_v) \in C(D'),
\end{salign}
which maps an orientation of the graph $G$ to the strongly connected component containing the vertex $v$. 
Then $\Phi_v(\cD_\rho)$ defines a random variable on $\Omega$, which we call the \textbf{strongly connected component of $\cD_\rho$ containing $v$}. 

Define the map
\begin{salign}
    \#\Phi_v : \Omega(G) &\to \R \\
    D' &\mapsto\# V_v,
\end{salign}
which maps an orientation of the graph $G$ to the size of the strongly connected component containing $v$. 
Then $\#\Phi_v(\cD_\rho)$ defines a random variable on $\Omega$, which we call the \textbf{size of the strongly connected component of $\cD_\rho$ containing $v$}.

Define the map 
    \begin{salign}
        \#\Phi_{}:\Omega &\to \R\\
                D' &\mapsto \max_{(V_i,E_i)\in C(D')} \# V_i.
    \end{salign}
Then $\#\Phi_{}(\cD_\rho)$ defines a random variable on $\Omega$, which we call the \textbf{size of the largest strongly connected component} of the randomly oriented graph $\cD_\rho$. 

We now consider the randomly oriented divisor graph $\cD_\rho(N)=(D_N,\rho)$. 
The main object of study in this article is the expected value of the largest strongly connected component of $\cD_\rho(N)$, which we denote by $\bE[\#\Phi_{}(\cD_\rho(N))]$. This can be explicitly computed for small values of $N$.

\begin{example}
    In this example we compute the expected value of the size of the largest strongly connected component of the randomly oriented divisor graph $\cD_\rho(5)$. 
    Consider the oriented divisor graph $D_5$, illustrated in Figure \ref{fig:5_divisor_graph}. 
    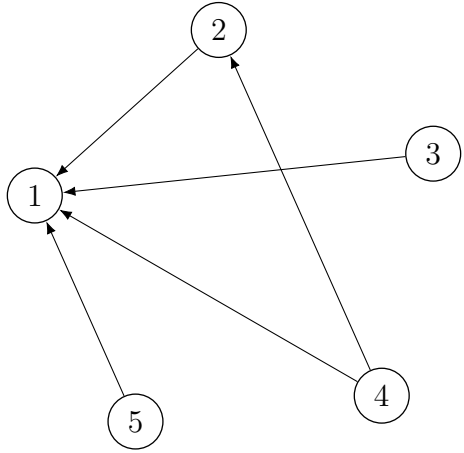
\begin{figure}[htbp]
        \centering
        \begin{tikzpicture}[scale=.7,-{Latex}]
            \SetGraphUnit{4}
            \GraphInit[vstyle=normal]
            \begin{scope}[rotate=240]
            \Vertices[Lpos=45]{circle}{5,4,3,2,1}
            \end{scope}
            \path (2) edge (1);
            \path (3) edge (1);
            \path (4) edge (1);
            \path (5) edge (1);
            \path (4) edge (2);
        \end{tikzpicture}
        \caption{Oriented divisor graph $D_5$.}
            \label{fig:5_divisor_graph} 
    \end{figure}
    In the corresponding randomly oriented graph $\cD_\rho(5)$, there are two possibilities for the size of the largest strongly connected component, $1$ or $3$. The latter case occurs precisely when the triangle with vertices $1,2,4$ is strongly connected. Finding the expected size of the largest strongly connected component is a straightforward computation:
    \begin{salign}
        \bE[\#\Phi_{}(\cD_\rho(5))]
       & = 3(\rho^2(1-\rho) + \rho(1-\rho)^2) + 1(1-(\rho^2(1-\rho) + \rho(1-\rho)^2))\\
        & = 3(\rho-\rho^2) + 1(1-\rho+\rho^2)\\
        & = 1 + 2\rho - 2\rho^2.
    \end{salign}
When $\rho=1/2$, we find that the average size of the largest strongly connected component is $\bE[\#\Phi_{}(\cD_{1/2}(5))]=3/2$.  

Additional examples of  $\bE[\#\Phi_{}(\cD_\rho(N))]$ for small $N$ can be found in Table \ref{tab:LSCC_polynomials}.
\end{example}

\begin{table}[h]
\centering
\renewcommand{\arraystretch}{1.2}
\begin{tabular}{cc}
\toprule
$N$ & $\bE[\#\Phi_{}(\cD_{\rho}(N))]$\\
\midrule
$1,2,3$ & $1$ \\
$4, 5$ & $1 + 2\rho - 2\rho^2$ \\
$6,7$ & $1 + 5\rho + 2\rho^2 - 18\rho^3 + 19\rho^4 - 12 \rho^5 + 4\rho^6$\\
$8$ & $1 + 10\rho - 4\rho^2 - 23\rho^3 + 43\rho^4 - 49\rho^5 + 35\rho^6 -16\rho^7 + 4\rho^8$\\
$9$ & $1 + 12\rho - 6\rho^2 - 18\rho^3 + 17\rho^4 + 10\rho^5 - 36\rho^6 + 28\rho^7 - 7\rho^8$\\
%
\bottomrule
\end{tabular}
\caption{Expected value of largest strongly connected component of $\cD_\rho(N)$ for small $N$.}
\label{tab:LSCC_polynomials}
\end{table}

\section{Proof of Theorem \ref{thm:Largest_SCC_abstract_bound}}\label{sec:LSCC_Thm}

In this section we prove Theorem \ref{thm:Largest_SCC_abstract_bound}.

\begin{proof}[Proof of Theorem \ref{thm:Largest_SCC_abstract_bound}.]
    As the size of the component containing $1$ is at most the size of the largest strongly connected component, $\bE[\#\Phi(\cD_\rho(N))]\geq \bE[\#\Phi_{1}(\cD_\rho(N))]$. 
    It therefore suffices to give a lower bound for $\bE[\#\Phi_{1}(\cD_\rho(N))]$. 
    
    Let $m\leq N$, and let 
    \begin{equation}
        d_0=1<d_1<\cdots<d_{\tau(m)-2}<d_{\tau(m)-1}=m
    \end{equation}
    denote the divisors of $m$. 

\begin{figure}[htbp]
    \centering
    \begin{tikzpicture}[
        vertex/.style={circle, draw, minimum size=1.2cm, inner sep=0pt, font=\large},
        mid arrow/.style={postaction={decorate,decoration={
            markings,
            mark=at position .55 with {\arrow{Stealth[length=3mm]}}
          }}}
    ]
    
        \node[vertex] (m) at (0, 3) {$m$};
        \node[vertex] (one) at (0, -3) {$1$};
    
        \node[vertex] (d1) at (-3.5, 0) {$d_1$};
        \node[vertex] (d2) at (-1.5, 0) {$d_2$};
        \node[font=\Large] (dots) at (.5, 0) {$\dots$};
        \node[vertex] (dt) at (2.5, 0) {$d_{\tau(m)-2}$};
    
        \draw[mid arrow] (m) -- (d1);
        \draw[mid arrow] (m) -- (d2);
        \draw[mid arrow] (m) -- (dt);
        
    
        \draw[mid arrow] (d1) -- (one);
        \draw[mid arrow] (d2) -- (one);
        \draw[mid arrow] (dt) -- (one);
        
    
        \draw[mid arrow] (m) to[out=180, in=180, looseness=2.5] (one);
    
    \end{tikzpicture}
    \caption{Subgraph of the oriented divisor graph consisting of vertices corresponding to the divisors of $m$.}
        \label{fig:m_divisor_subgraph} 
\end{figure}
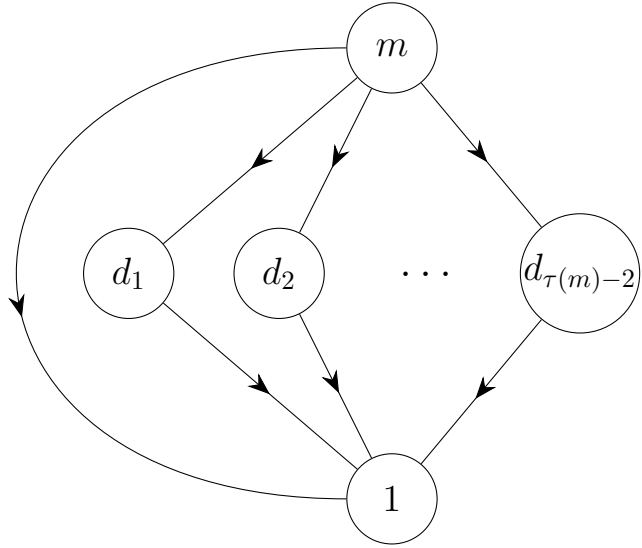
We compute the probability that $m$ is in the same strongly connected component as $1$. We have 
\begin{equation}
    \bP\left[m\in \Phi_1(\cD_\rho(N))\right]
    = 1 - \bP\left[m\not\in \Phi_1(\cD_\rho(N))\right].
\end{equation}
Note that $\bP\left[m\not\in \Phi_1(\cD_\rho(N))\right]$ is bounded above by the probability that none of the triangles with vertices $1$, $d_i$, and $m$ are strongly connected. Conditioning on the direction of the edge between $1$ and $m$, the triangles involve distinct remaining edges and are therefore independent. We use this observation to compute the probability that none of these triangles are strongly connected,
\begin{salign}
      & \bP\left[(1,m)\in \cD_\rho(N)\right] \prod_{i=1}^{\tau(m)-2} \left(1 - \bP\left[(m,d_i), (d_i, 1)\in \cD_\rho(N)\right]\right)\\
      &\hspace{2cm} + \bP\left[(m,1)\in \cD_\rho(N)\right] \prod_{i=1}^{\tau(m)-2} \left(1 - \bP\left[(d_i,m), (1, d_i)\in \cD_\rho(N)\right]\right)\\
      &\ = \rho (1-(1-\rho)^2)^{\tau(m)-2} + (1-\rho)(1-\rho^2)^{\tau(m)-2}\\
      &\ = \rho(2\rho-\rho^2)^{\tau(m)-2} + (1-\rho)(1-\rho^2)^{\tau(m)-2}.
\end{salign}
Therefore
\begin{equation}\label{eq:m_in_component_1}
    \bP\left[m\in \Phi_1(\cD_\rho(N))\right] \geq 1 - \rho(2\rho-\rho^2)^{\tau(m)-2} - (1-\rho)(1-\rho^2)^{\tau(m)-2}.
\end{equation}

 By the bound (\ref{eq:m_in_component_1}) and the assumption that there are at least $y$ integers $m\leq N$ with $\tau(m)-2\geq x$,
\begin{equation}
    \bE[\#\Phi_{1}(\cD_\rho(N))]
    \geq \sum_{n\leq N}\bP\left[n\in \Phi_1(\cD_\rho(N)\right]
    \geq y \left(1 - \rho(2\rho-\rho^2)^{x} - (1-\rho)(1-\rho^2)^{x}\right).
\end{equation}
\end{proof}

\begin{remark}
    In the above proof we only consider paths between $1$ and $m$ of length at most $2$. It would be interesting to consider longer paths, perhaps using inclusion-exclusion or a recursive argument, which would strengthen the bound \eqref{eq:LSCC_bound} in Theorem \ref{thm:Largest_SCC_abstract_bound}.
\end{remark}

\section{Effective version of a result of Hardy and Ramanujan}\label{sec:divisors}


In this section we prove Theorem \ref{thm:divisors_Durkan} and Proposition \ref{prop:primorial_bound}. 

Recall that $\tau(n)$ denotes the number of positive divisors of $n$ (e.g., $\tau(10)=\#\{1,2,5,10\}=4$).
The number of edges of the divisor graph $G_N$ is
\begin{equation}
    \sum_{n=1}^N (\tau(n)-1) = -N + \sum_{n=1}^N \tau(n).
\end{equation}
Let $\gamma$ denote the Euler--Mascheroni constant.
In 1849 Dirichlet \cite{Dir51} proved that
\begin{equation}
    \sum_{n=1}^N \tau(n) = N\log(N) + (2\gamma - 1)N + O(\sqrt{N}),
\end{equation}
 (see also \cite[page 22]{IK04}).
From this we obtain an approximation for the average degree of a vertex in $G_N$.

\begin{proposition}
    The average degree of a vertex in the divisor graph $G_N$ is
    \begin{equation}
        2\log(N) + 2(2\gamma - 3) + O(N^{-1/2}).
    \end{equation}
\end{proposition}

\begin{definition}\label{def:normal_order}
    Let $f:\Z_{>0}\to \C$ and $g:\Z_{>0}\to \C$ be arithmetic function. Then $g$ is a \textbf{normal order} of $f$ if for each $\epsilon\in \R_{>0}$ there exists a subset $S_{\epsilon}\subset \Z_{\geq 1}$ of density one\footnote{That is, $\lim_{n\to\infty} \frac{\#\{s\in S_{\epsilon} : s\leq n\}}{n}=1$.} such that for each $n\in S_{\epsilon}$
    \begin{equation}
        (1-\epsilon) g(n)\leq f(n) \leq (1+\epsilon)g(n).
    \end{equation}
\end{definition} 

A consequence of a result of Hardy and Ramanujan \cite{HR17} shows that $\log(\tau(n))$ has normal order $\log(2)\log\log(n)$. Theorem \ref{thm:divisors_Durkan} proves an effective version of this result.

Let $\omega(n)$ denote the number of distinct prime divisors of $n$ (e.g., $\omega(12)=\#\{2,3\}=2$). Note that $\tau(m)\geq 2^{\omega(m)}$. 
The proof of Theorem \ref{thm:divisors_Durkan} makes use of a recent effective version of the Hardy--Ramanujan theorem due to Durkan \cite{Dur23}.

\begin{lemma}[{\cite[Theorem 1.2]{Dur23}}]\label{lem:durkan_level}
    For every $N\geq 3$ and every integer $k\geq 1$,
    \begin{equation}\label{eq:durkan_level}
        \#\{n\leq N : \omega(n)=k\} \leq 1.95\,\frac{N}{\log(N)}\,\frac{\left(\log\log(N)+4.096\right)^{k-1}}{(k-1)!}.
    \end{equation}
\end{lemma}

\begin{proof}[Proof of Theorem \ref{thm:divisors_Durkan}.]
    Since $\log(N)^{\log(2)(1-\epsilon)}=2^{(1-\epsilon)\log\log(N)}$ and $\tau(n)\geq 2^{\omega(n)}$, any $n\leq N$ with $\tau(n)<\log(N)^{\log(2)(1-\epsilon)}$ satisfies $\omega(n)<(1-\epsilon)\log\log(N)$. Writing $E$ for the number of such $n$, the left-hand side of \eqref{eq:divisors_Durkan} is at least $N-E$, so it suffices to show that $E\leq c_\epsilon N/\log(N)^{\epsilon^2/2}$.

    The integer $n=1$ contributes $1$ to $E$, and we bound the remaining $n$ using Lemma \ref{lem:durkan_level}. Defining $K\defeq (1-\epsilon)(\log\log(N)+4.096)$ and re-indexing by $j=k-1$, by Lemma \ref{lem:durkan_level} we have
    \begin{equation}\label{eq:sum_levels}
        E-1 \;\leq\; \frac{1.95\,N}{\log(N)}\sum_{0\leq j<K}\frac{(\log\log(N)+4.096)^{j}}{j!}.
    \end{equation}
    The terms in this sum grow as a function of $j$, and since $j<K$, the ratios of consecutive terms are
    \begin{equation}
        \frac{\log\log(N)+4.096}{j+1} \geq 
        \frac{\log\log(N)+4.096}{K}
        =\frac{1}{1-\epsilon}.
    \end{equation}
    Therefore the entire sum is at most its final term times 
    \begin{equation}
        \sum_{i\in \Z_{\geq 0}}(1-\epsilon)^{i}=\frac{1}{\epsilon}.
    \end{equation}
    Since the final term is bounded above by
    \begin{equation}
        \frac{(\log\log(N)+4.096)^{K}}{K!},
    \end{equation}
    we obtain the bound
    \begin{equation}\label{eq:E-1_bound}
        E-1 \;<\; \frac{1.95}{\epsilon}\,\frac{N}{\log(N)}\frac{(\log\log(N)+4.096)^K}{K!}.
    \end{equation}
    From the Taylor series for $e^x$ we have $K!>(K/e)^{K}$.  Using this bound and the definition of $K$, the upper bound \eqref{eq:E-1_bound} becomes
    \begin{equation}\label{eq:collapse}
        E-1 \;<\; \frac{1.95}{\epsilon}\,\frac{N}{\log(N)}\left(\frac{e(\log\log(N)+4.096)}{K}\right)^{K}
        = \frac{1.95}{\epsilon}\,\frac{N}{\log(N)}\left(\frac{e}{1-\epsilon}\right)^{K}.
    \end{equation}
    The final factor is
    \begin{salign}\label{eq:e/(1-epsilon)^K_bound}
        \left(\frac{e}{1-\epsilon}\right)^{K}
        &=\left(\frac{e}{1-\epsilon}\right)^{(1-\epsilon)4.096}\left(\frac{e}{1-\epsilon}\right)^{(1-\epsilon)\log\log(N)}\\
        &\leq\left(\frac{e}{1-\epsilon}\right)^{4.096}\left(\frac{e}{1-\epsilon}\right)^{(1-\epsilon)\log\log(N)}\\
        &=\left(\frac{e}{1-\epsilon}\right)^{4.096}\log(N)^{1-\epsilon-(1-\epsilon)\log(1-\epsilon)}.
    \end{salign}
    Plugging this bound back into \eqref{eq:collapse}, and recalling the definition of $c_\epsilon$ from equation \eqref{eq:c_epsilon_def}, we have
    \begin{salign}\label{eq:E-1_final}
        E-1 &< \frac{1.95}{\epsilon}\left(\frac{e}{1-\epsilon}\right)^{4.096}\frac{N}{\log(N)^{\epsilon+(1-\epsilon)\log(1-\epsilon)}} \\
        &=\; \frac{1.95}{4}\,c_\epsilon\,\frac{N}{\log(N)^{\epsilon+(1-\epsilon)\log(1-\epsilon)}} \\
        &< \frac{c_\epsilon}{2}\cdot\frac{N}{\log(N)^{\epsilon+(1-\epsilon)\log(1-\epsilon)}}.
    \end{salign}
    Using the Taylor expansion of $\log(1-x)$, we have that $\epsilon+(1-\epsilon)\log(1-\epsilon)>\epsilon^2/2$, and applying this to \eqref{eq:E-1_final} gives
    \begin{equation}\label{eq:E-1_power}
        E-1 \;<\; \frac{c_\epsilon}{2}\cdot\frac{N}{\log(N)^{\epsilon^2/2}}.
    \end{equation}
    Since $\log(N)^{\epsilon^2/2}\leq N$ for $N\geq 3$, we have $N/\log(N)^{\epsilon^2/2}\geq 1$. This, combined with  bound $c_\epsilon>4e^{4.096}>1$, shows that
    \begin{equation}
        E \;<\; 1+\frac{c_\epsilon}{2}\cdot\frac{N}{\log(N)^{\epsilon^2/2}} \;\leq\; \frac{c_\epsilon}{2}\cdot\frac{N}{\log(N)^{\epsilon^2/2}}+\frac{c_\epsilon}{2}\cdot\frac{N}{\log(N)^{\epsilon^2/2}} \;=\; \frac{c_\epsilon N}{\log(N)^{\epsilon^2/2}}.
    \end{equation}
    Subtracting from $N$ gives \eqref{eq:divisors_Durkan}.
\end{proof}

We now prove Corollary \ref{cor:Explicit_Largest_SCC_large_N}.

\begin{proof}[Proof of Corollary \ref{cor:Explicit_Largest_SCC_large_N}.]
    By Theorem \ref{thm:divisors_Durkan} we may apply Theorem \ref{thm:Largest_SCC_abstract_bound} with $x=\log(N)^{\log(2)(1-\epsilon)}-2$ and $y=N\left(1-c_\epsilon/\log(N)^{\epsilon^2/2}\right)$, which proves the corollary.
\end{proof}

Although Theorem \ref{thm:divisors_Durkan} holds for all $N\geq 3$, the constant $c_\epsilon$ renders it vacuous unless $N$ is large in terms of $\epsilon$. To contrast this, we give a simple, non-trivial bound that holds for all $N$. 

Let $p_m$ denote the $m$-th prime number (e.g., $p_3=5$). The $n$-th primorial is defined as the product of the first $n$ primes. 
\begin{equation}
    p_n\prim \defeq \prod_{m=1}^n p_m.
\end{equation}

\begin{proposition}\label{prop:primorial_bound}
Let $D\in \R_{\geq 1}$, and let $d=\lceil \log_2(D)\rceil$. Then
    \begin{equation}\label{eq:prim_bound}
        \#\{n\leq N : \tau(n) \geq D\} \geq \left\lfloor \frac{N}{p_{d}\prim}\right\rfloor.
    \end{equation}
\end{proposition}

\begin{proof}
    Note that $\tau(p_d\prim)=2^d\geq D$. Each  element of the set
    \begin{equation}
        \{ m\cdot p_d\prim : m\in \Z_{\geq 1} \text{ and } m\cdot p_d\prim\leq N\}
        =
         \{ m\cdot p_d\prim : 1 \leq m\leq \lfloor N/p_d\prim\rfloor\}
    \end{equation}
    has $\tau(n\cdot p_d\prim)\geq \tau(p_d\prim)\geq D$. As the set has cardinality $\lfloor N/p_d\prim\rfloor$, we obtain the bound (\ref{eq:prim_bound}). 
\end{proof}


\section{Simulations}\label{sec:simulations}

In this section we discuss data from simulations for the largest strongly connected component and diameter of randomly oriented divisor graphs. The data and code used to produce the data can be found on the \texttt{GitHub} repository \cite{KP26}. 

For the largest strongly connected component, we compute data for $N\in \{256m : 1 \leq m\leq 1024\}$  and $\rho\in \{0.1, 0.2, 0.3, 0.4, 0.5\}$. For each pair $(N,\rho)$ we sample $50$ directed graphs from $\cD_\rho(N)$, and for each of these sampled graphs we compute the size of their largest strongly connected component using Tarjan's algorithm. Figure \ref{fig:largest_scc} displays the ratio of the average largest strongly connected component over the samples for each $(N,\rho)$ with $N$. Note that what is displayed is a discrete set points, and not actually a line. Figure \ref{fig:largest_scc_0.5} 
displays this ratio when $\rho=0.5$ together with the lower bound coming from Corollary \ref{cor:Explicit_Largest_SCC_Primorial}. We see that the lower bound from Corollary \ref{cor:Explicit_Largest_SCC_Primorial} shows that on average the largest strongly connected component makes up a positive proportion of the entire graph, but that this proportion is much less than the true proportion. 

\begin{figure}[htbp]
  \centering
  \includegraphics[scale=0.4]{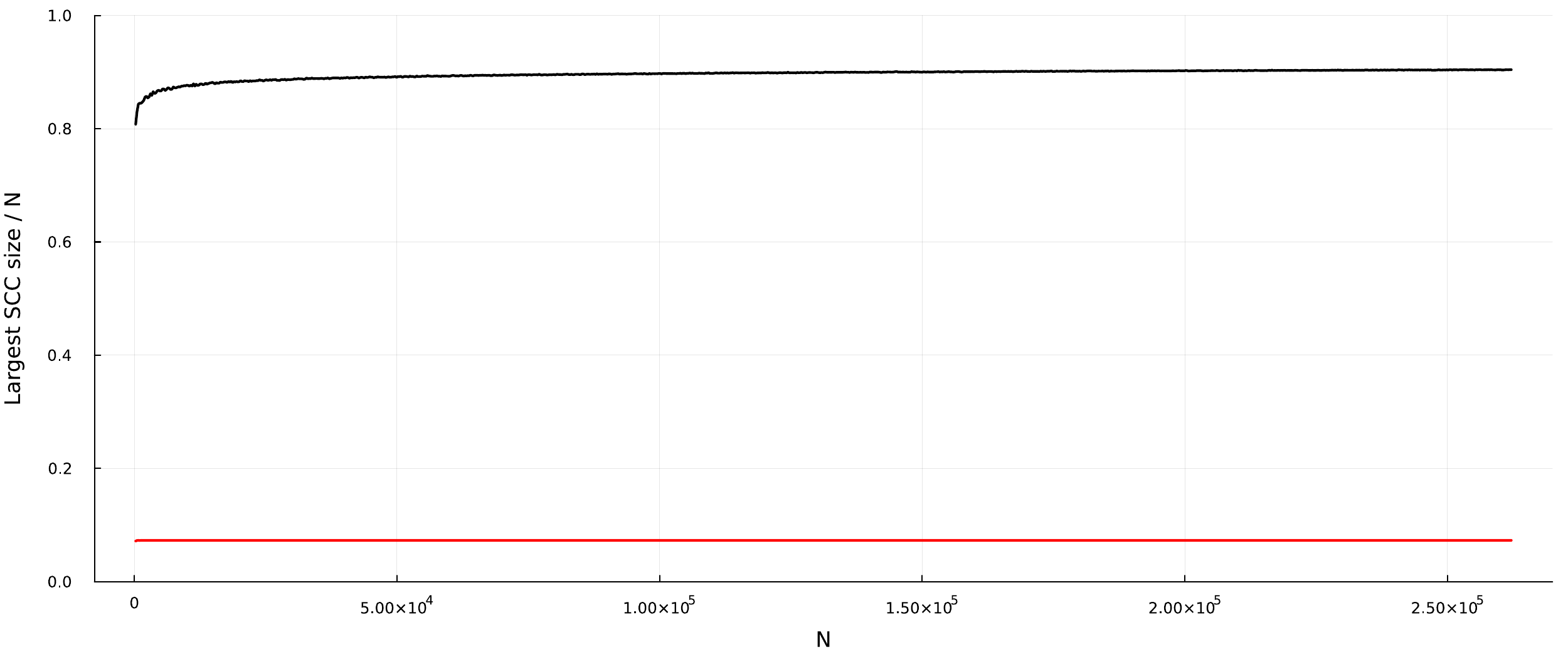}
  \caption{Average size of largest strongly connected component (SCC) with $\rho=0.5$ (in black) and lower bound from Corollary \ref{cor:Explicit_Largest_SCC_Primorial} (in red). }
  \label{fig:largest_scc_0.5}
\end{figure}

\begin{figure}[htbp]
  \centering
  \includegraphics[scale=0.4]{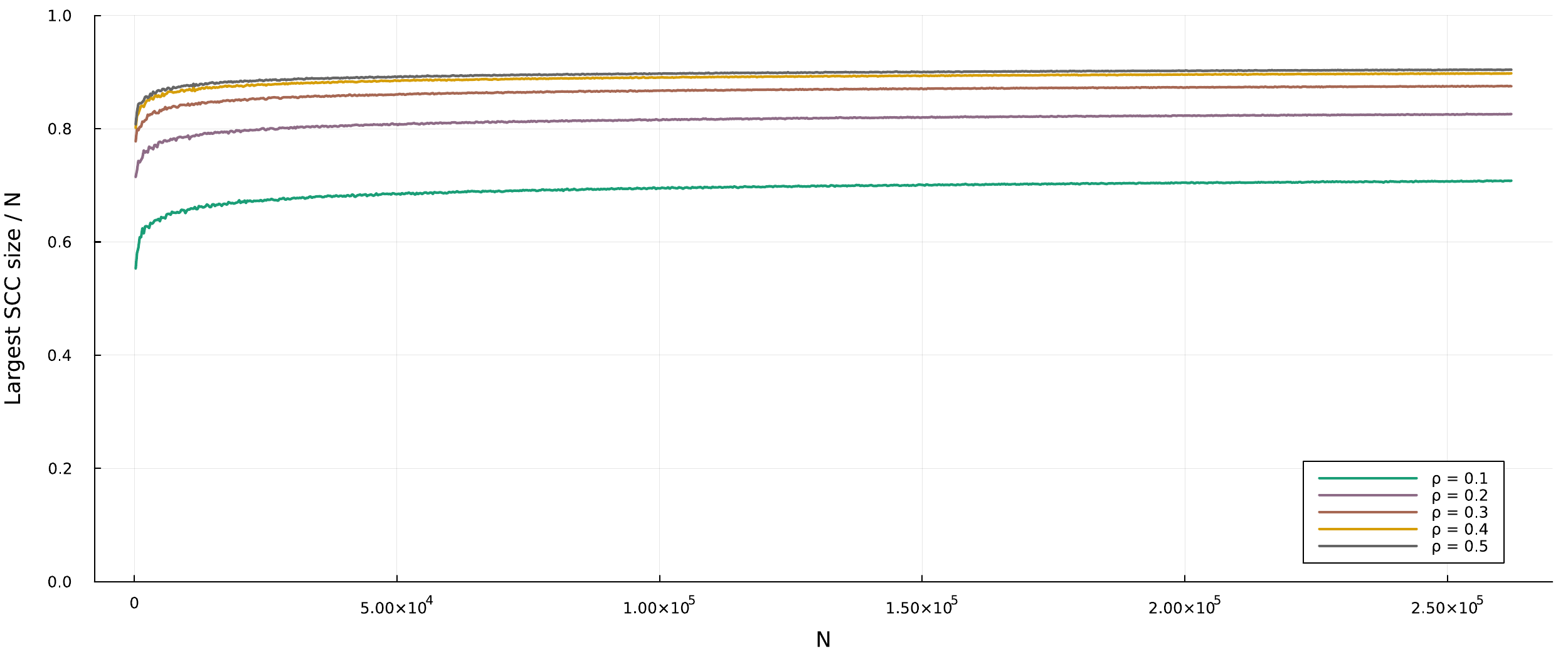}
  \caption{Average size of largest strongly connected component (SCC) for different $\rho$ values. }
  \label{fig:largest_scc}
\end{figure}

Recall that the diameter of a directed graph is the maximum shortest path between two vertices. Since most of the graphs we consider will not be strongly connected, there will not be a direct path between every two vertices. For this reason, the diameter will always refer to the diameter of the largest strongly connected component.

For the diameter, we compute data for $N\in \{1024m : 1 \leq m\leq 323\}$  and $\rho\in \{0.1, 0.2, 0.3, 0.4, 0.5\}$. For each pair $(N,\rho)$ we sample $10$ directed graphs from $\cD_\rho(N)$, and for each of these sampled graphs we compute the diameter using the algorithm in \cite{CGLM12}.\footnote{Computing the diameter is much more computationally intensive than the largest strongly connected component, which is why we computed fewer samples.} 
Figure \ref{fig:diameter_0.5} displays the average diameter over the samples for each $N$ with $\rho=0.5$, together with a line of best fit obtained from linear regression. 

\begin{figure}[htbp]
  \centering
  \includegraphics[scale=0.4]{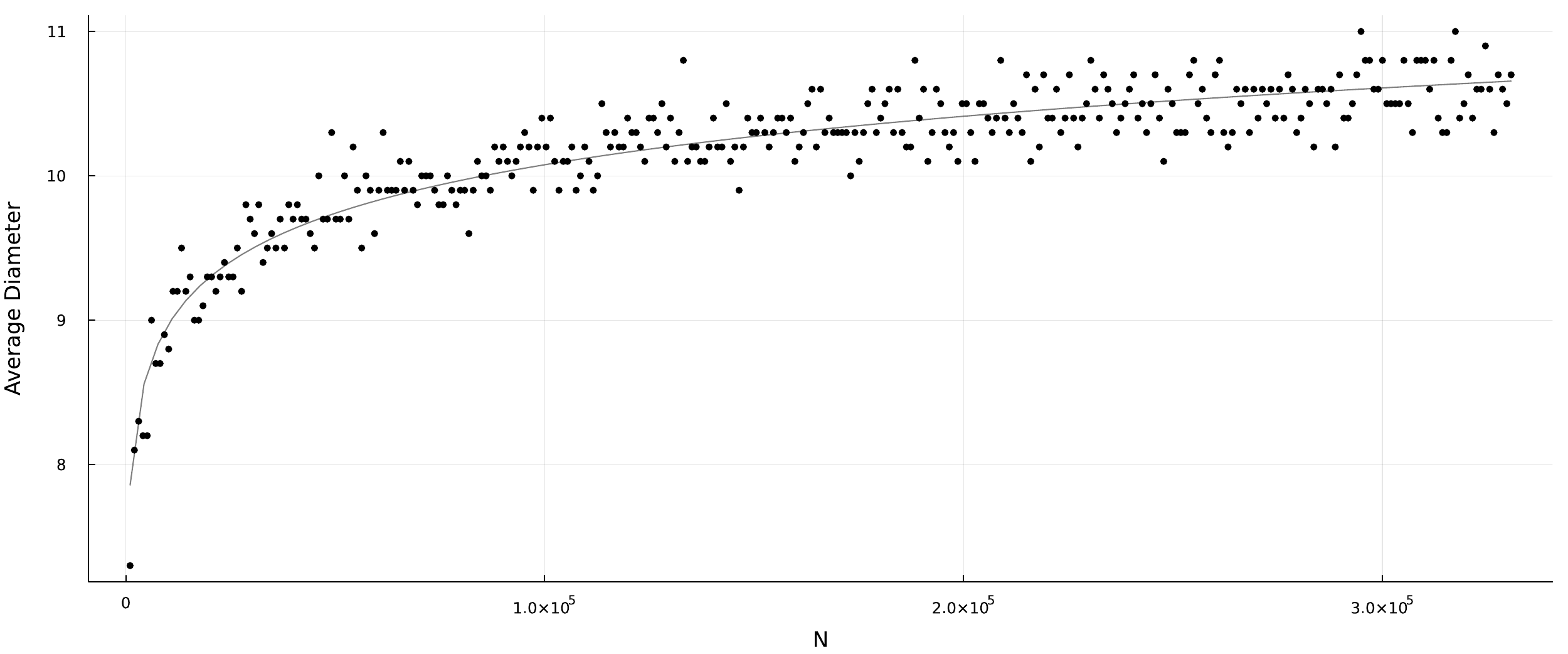}
  \caption{Average diameter of simulation with $\rho=0.5$, together with line of best fit. }
  \label{fig:diameter_0.5}
\end{figure}

For each value of $\rho$, using linear regression we find constants $\alpha_\rho$ and $\beta_\rho$ (given in Table \ref{tab:diameter}) such that the average diameter of $\cD_\rho(N)$ is approximately
\begin{equation}
    \alpha_\rho \log(N) + \beta_\rho.
\end{equation}
This gives computation support to Conjecture \ref{conj:diameter}.

\begin{table}[h]
\centering
\renewcommand{\arraystretch}{1.2}
\begin{tabular}{cccc}
\toprule
$\rho$ & $\alpha_\rho$ & $\beta_\rho$ & Mean Squared Error\\
\midrule
$0.1,\ 0.9$ & $0.631556$  & $3.9032$ & $0.0533449$ \\
$0.2,\ 0.8$ & $0.522757$ & $4.5697$ & $0.0414596$\\
$0.3,\ 0.7$ & $0.527978$ & $4.2336$ & $0.0409155$ \\
$0.4,\ 0.6$ & $0.500323$ & $4.3819$ & $0.028766$\\
$0.5$ & $0.484408$ & $4.4997$ & $0.0323064$ \\
\bottomrule
\end{tabular}
\caption{Lines of best fit for average diameter.}
\label{tab:diameter}
\end{table}

\begin{remark}
    Although we put forth Conjecture \ref{conj:diameter} and give some computation support for it, we acknowledge that there are significant computational limitations. For example, the size of the diameter may grow much slower, like $\log\log(N)$, and this may not be apparent in the data for the relatively small values of $N$ we are able to compute. The result of Hardy and Ramanujan on the normal order of $\log(\tau(n))$ is a good example of a result that is infeasible to check empirically.  
\end{remark}

\section{Further questions}\label{sec:questions}

In this section we pose some further questions about randomly oriented divisor graphs.

\begin{question}
    From the data in Section \ref{sec:simulations} we see that the lower bound in Corollary \ref{cor:Explicit_Largest_SCC_Primorial} is far from optimal. To what extent can this lower bound be improved for values of $N$ for which Corollary \ref{cor:Explicit_Largest_SCC_large_N} does not apply?
\end{question}

\begin{question}
    What can be said about the expected size of the largest strongly connected component as we remove the vertices of highest degrees, one by one. In other words, how will the expected size of the largest strongly connected component change as we remove the vertices labeled 1, 2, 3, 4, and so on?\footnote{This question is motivated by hub removal attacks in network science.}
\end{question}

\begin{question}
    What can be said about the variance of the largest strongly connected component in randomly oriented divisor graphs?\\
    Can a suitable version of the Erd\H{o}s--Kac Theorem be used to understand the limiting distribution of $\#\Phi_{}(\cD_\rho(N))$ as $N\to\infty$?
\end{question}

\begin{question}
    What is the expected number of strongly connected triangles in the randomly oriented divisor graph $\cD_\rho(N)$?
\end{question}

\begin{question}
    To what extent do randomly oriented divisor graphs behave similarly to (or differently from) real, scale-free networks?
\end{question}

\bibliographystyle{alpha}
\bibliography{bibfile}

@article {AAA10,
    AUTHOR = {Al-Addasi, S. and AbuGhneim, O. A. and Al-Ezeh, H.},
     TITLE = {Characterizing powers of cycles that are divisor graphs},
   JOURNAL = {Ars Combin.},
  FJOURNAL = {Ars Combinatoria. A Canadian Journal of Combinatorics},
    VOLUME = {97A},
      YEAR = {2010},
     PAGES = {447--451},
      ISSN = {0381-7032,2817-5204},
   MRCLASS = {05C76 (05C38)},
  MRNUMBER = {2721819},
}

@incollection {CP04 ,
    AUTHOR = {Campanino, Massimo and Petritis, Dimitri},
     TITLE = {On the physical relevance of random walks: an example of
              random walks on a randomly oriented lattice},
 BOOKTITLE = {Random walks and geometry},
     PAGES = {393--411},
 PUBLISHER = {Walter de Gruyter, Berlin},
      YEAR = {2004},
      ISBN = {3-11-017237-2},
   MRCLASS = {60G50},
  MRNUMBER = {2087791},
MRREVIEWER = {Heinrich\ Niederhausen},
}

@article{CGPS11,
    AUTHOR = {Castell, Fabienne and Guillotin-Plantard, Nadine and P\`ene,
              Fran\c coise and Schapira, Bruno},
     TITLE = {A local limit theorem for random walks in random scenery and
              on randomly oriented lattices},
   JOURNAL = {Ann. Probab.},
  FJOURNAL = {The Annals of Probability},
    VOLUME = {39},
      YEAR = {2011},
    NUMBER = {6},
     PAGES = {2079--2118},
      ISSN = {0091-1798,2168-894X},
   MRCLASS = {60K37 (60F05 60G50 60G52)},
  MRNUMBER = {2932665},
MRREVIEWER = {Andrew\ R.\ Wade},
       DOI = {10.1214/10-AOP606},
       URL = {https://doi-org.dartmouth.idm.oclc.org/10.1214/10-AOP606},
}

@inproceedings {CMS01,
    AUTHOR = {Chartrand, Gary and Muntean, Raluca and Saenpholphat, Varaporn
              and Zhang, Ping},
     TITLE = {Which graphs are divisor graphs?},
 BOOKTITLE = {Proceedings of the {T}hirty-second {S}outheastern
              {I}nternational {C}onference on {C}ombinatorics, {G}raph
              {T}heory and {C}omputing ({B}aton {R}ouge, {LA}, 2001)},
   JOURNAL = {Congr. Numer.},
  FJOURNAL = {Congressus Numerantium. A Conference Journal on Numerical
              Themes},
    VOLUME = {151},
      YEAR = {2001},
     PAGES = {189--200},
      ISSN = {0384-9864},
   MRCLASS = {05C78 (05C12 05C20)},
  MRNUMBER = {1887439},
}

@InProceedings{CGLM12,
author="Crescenzi, Pierluigi
and Grossi, Roberto
and Lanzi, Leonardo
and Marino, Andrea",
editor="Klasing, Ralf",
title="On Computing the Diameter of Real-World Directed (Weighted) Graphs",
booktitle="Experimental Algorithms",
year="2012",
publisher="Springer Berlin Heidelberg",
address="Berlin, Heidelberg",
pages="99--110",
isbn="978-3-642-30850-5"
}

@article{Dir51,
  author  = {Dirichlet, Peter Gustav Lejeune},
  title   = {{\"U}ber die Bestimmung der mittleren Werthe in der Zahlentheorie},
  journal = {Abhandlungen der K{\"o}niglichen Preu{\ss}ischen Akademie der Wissenschaften zu Berlin},
  year    = {1851},
  note    = {Read 1849},
  pages   = {69--83}
}

@misc{Dur23,
    AUTHOR = {Durkan, Benjamin},
     TITLE = {On the {H}ardy--{R}amanujan theorem},
      YEAR = {2023},
    EPRINT = {2310.14760},
 ARCHIVEPREFIX = {arXiv},
 PRIMARYCLASS = {math.NT},
      NOTE = {arXiv:2310.14760},
       DOI = {10.48550/arXiv.2310.14760},
       URL = {https://arxiv.org/abs/2310.14760},
}

@article {ES95,
    AUTHOR = {Erd\H{o}s, Paul and Saias, \'Eric},
     TITLE = {Sur le graphe divisoriel},
   JOURNAL = {Acta Arith.},
  FJOURNAL = {Acta Arithmetica},
    VOLUME = {73},
      YEAR = {1995},
    NUMBER = {2},
     PAGES = {189--198},
      ISSN = {0065-1036,1730-6264},
   MRCLASS = {11N37 (05C38)},
  MRNUMBER = {1358194},
MRREVIEWER = {A.\ J.\ Hildebrand},
       DOI = {10.4064/aa-73-2-189-198},
       URL = {https://doi-org.dartmouth.idm.oclc.org/10.4064/aa-73-2-189-198},
}

@article{Fra03,
  author  = {Frayer, Christopher},
  title   = {Properties of Divisor Graphs},
  journal = {Rose-Hulman Undergraduate Mathematics Journal},
  volume  = {4},
  number  = {2},
  article = {4},
  year    = {2003}
}

@article {Gri01,
    AUTHOR = {Grimmett, Geoffrey R.},
     TITLE = {Infinite paths in randomly oriented lattices},
   JOURNAL = {Random Structures Algorithms},
  FJOURNAL = {Random Structures \& Algorithms},
    VOLUME = {18},
      YEAR = {2001},
    NUMBER = {3},
     PAGES = {257--266},
      ISSN = {1042-9832,1098-2418},
   MRCLASS = {05D40 (60K35 82B43)},
  MRNUMBER = {1824275},
       DOI = {10.1002/rsa.1007},
       URL = {https://doi-org.dartmouth.idm.oclc.org/10.1002/rsa.1007},
}

@article {GL07,
    AUTHOR = {Guillotin-Plantard, N. and Le Ny, A.},
     TITLE = {Transient random walks on 2{D}-oriented lattices},
   JOURNAL = {Teor. Veroyatn. Primen.},
  FJOURNAL = {Teoriya Veroyatnoste\u i\ i ee Primeneniya},
    VOLUME = {52},
      YEAR = {2007},
    NUMBER = {4},
     PAGES = {815--826},
      ISSN = {0040-361X,2305-3151},
   MRCLASS = {60G50},
  MRNUMBER = {2742878},
       DOI = {10.1137/S0040585X97983353},
       URL = {https://doi-org.dartmouth.idm.oclc.org/10.1137/S0040585X97983353},
}

@incollection {HR17,
    AUTHOR = {Hardy, G. H. and Ramanujan, S.},
     TITLE = {The normal number of prime factors of a number {$n$} [{Q}uart.
              {J}. {M}ath. {\bf 48} (1917), 76--92]},
 BOOKTITLE = {Collected papers of {S}rinivasa {R}amanujan},
     PAGES = {262--275},
 PUBLISHER = {AMS Chelsea Publ., Providence, RI},
      YEAR = {2000},
      ISBN = {0-8218-2076-1},
   MRCLASS = {01A75},
  MRNUMBER = {2280878},
}

@book {IK04,
    AUTHOR = {Iwaniec, Henryk and Kowalski, Emmanuel},
     TITLE = {Analytic number theory},
    SERIES = {American Mathematical Society Colloquium Publications},
    VOLUME = {53},
 PUBLISHER = {American Mathematical Society, Providence, RI},
      YEAR = {2004},
     PAGES = {xii+615},
      ISBN = {0-8218-3633-1},
   MRCLASS = {11-02 (11Fxx 11Lxx 11Mxx 11Nxx)},
  MRNUMBER = {2061214},
MRREVIEWER = {K.\ Soundararajan},
       DOI = {10.1090/coll/053},
       URL = {https://doi-org.dartmouth.idm.oclc.org/10.1090/coll/053},
}

@misc{KP26,
  author = {Kim, Jihyung and Phillips, Tristan},
  title = {randomly-oriented-divisor-graphs},
  year = {2026},
  url = {https://github.com/LoveLow-Global/randomly-oriented-divisor-graphs},
  note = {GitHub repository: {https://github.com/LoveLow-Global/randomly-oriented-divisor-graphs}}
}

@article {Lin11,
    AUTHOR = {Linusson, Svante},
     TITLE = {On percolation and the bunkbed conjecture},
   JOURNAL = {Combin. Probab. Comput.},
  FJOURNAL = {Combinatorics, Probability and Computing},
    VOLUME = {20},
      YEAR = {2011},
    NUMBER = {1},
     PAGES = {103--117},
      ISSN = {0963-5483,1469-2163},
   MRCLASS = {05C80 (05C20 05C76)},
  MRNUMBER = {2745680},
MRREVIEWER = {Raphael\ Yuster},
       DOI = {10.1017/S0963548309990666},
       URL = {https://doi-org.dartmouth.idm.oclc.org/10.1017/S0963548309990666},
}

@article {McD81,
    AUTHOR = {McDiarmid, Colin},
     TITLE = {General percolation and random graphs},
   JOURNAL = {Adv. in Appl. Probab.},
  FJOURNAL = {Advances in Applied Probability},
    VOLUME = {13},
      YEAR = {1981},
    NUMBER = {1},
     PAGES = {40--60},
      ISSN = {0001-8678,1475-6064},
   MRCLASS = {60K99 (05C99 82A42)},
  MRNUMBER = {595886},
MRREVIEWER = {G.\ R.\ Grimmett},
       DOI = {10.2307/1426466},
       URL = {https://doi-org.dartmouth.idm.oclc.org/10.2307/1426466},
}

@article {McN21,
    AUTHOR = {McNew, Nathan},
     TITLE = {Counting primitive subsets and other statistics of the divisor
              graph of {$\{1,2,\dots,n\}$}},
   JOURNAL = {European J. Combin.},
  FJOURNAL = {European Journal of Combinatorics},
    VOLUME = {92},
      YEAR = {2021},
     PAGES = {Paper No. 103237, 20},
      ISSN = {0195-6698,1095-9971},
   MRCLASS = {05C30 (05C25)},
  MRNUMBER = {4149161},
MRREVIEWER = {Ali\ Reza\ Moghaddamfar},
       DOI = {10.1016/j.ejc.2020.103237},
       URL = {https://doi-org.dartmouth.idm.oclc.org/10.1016/j.ejc.2020.103237},
}

@article {MS20,
    AUTHOR = {Melotti, Paul and Saias, \'Eric},
     TITLE = {On path partitions of the divisor graph},
   JOURNAL = {Acta Arith.},
  FJOURNAL = {Acta Arithmetica},
    VOLUME = {192},
      YEAR = {2020},
    NUMBER = {4},
     PAGES = {329--339},
      ISSN = {0065-1036,1730-6264},
   MRCLASS = {11N37 (05C38 05C70 11B75)},
  MRNUMBER = {4054577},
MRREVIEWER = {John\ J.\ Watkins},
       DOI = {10.4064/aa180711-26-4},
       URL = {https://doi-org.dartmouth.idm.oclc.org/10.4064/aa180711-26-4},
}

@article {Nar18,
    AUTHOR = {Narayanan, Bhargav},
     TITLE = {Connections in randomly oriented graphs},
   JOURNAL = {Combin. Probab. Comput.},
  FJOURNAL = {Combinatorics, Probability and Computing},
    VOLUME = {27},
      YEAR = {2018},
    NUMBER = {4},
     PAGES = {667--671},
      ISSN = {0963-5483,1469-2163},
   MRCLASS = {60C05 (05C20 05C38 60K35)},
  MRNUMBER = {3816063},
MRREVIEWER = {Christian\ M\"onch},
       DOI = {10.1017/S0963548316000341},
       URL = {https://doi-org.dartmouth.idm.oclc.org/10.1017/S0963548316000341},
}

@inproceedings {Pom83,
    AUTHOR = {Pomerance, Carl},
     TITLE = {On the longest simple path in the divisor graph},
 BOOKTITLE = {Proceedings of the fourteenth {S}outheastern conference on
              combinatorics, graph theory and computing ({B}oca {R}aton,
              {F}la., 1983)},
   JOURNAL = {Congr. Numer.},
  FJOURNAL = {Congressus Numerantium. A Conference Journal on Numerical
              Themes},
    VOLUME = {40},
      YEAR = {1983},
     PAGES = {291--304},
      ISSN = {0384-9864},
   MRCLASS = {05C38 (11N37)},
  MRNUMBER = {734378},
MRREVIEWER = {A.\ D.\ Pollington},
}

@article{RA20,
  title={Patterns of primes and composites from divisibility network of natural numbers},
  author={Rajans, Abiya and Ambika, G.},
  journal={Network Science},
  volume={8},
  number={4},
  pages={519--532},
  year={2020},
  publisher={Cambridge University Press},
  doi={10.1017/nws.2020.20}
}

@article {RD23,
    AUTHOR = {Ravi, Vignesh and Desikan, Kalyani},
     TITLE = {Brief survey on divisor graphs and divisor function graphs},
   JOURNAL = {AKCE Int. J. Graphs Comb.},
  FJOURNAL = {AKCE International Journal of Graphs and Combinatorics},
    VOLUME = {20},
      YEAR = {2023},
    NUMBER = {2},
     PAGES = {217--225},
      ISSN = {0972-8600,2543-3474},
   MRCLASS = {05C76 (05-02)},
  MRNUMBER = {4637194},
       DOI = {10.1080/09728600.2023.2234979},
       URL = {https://doi-org.dartmouth.idm.oclc.org/10.1080/09728600.2023.2234979},
}

@article {Sai98,
    AUTHOR = {Saias, Eric},
     TITLE = {Applications des entiers \`a{} diviseurs denses},
   JOURNAL = {Acta Arith.},
  FJOURNAL = {Acta Arithmetica},
    VOLUME = {83},
      YEAR = {1998},
    NUMBER = {3},
     PAGES = {225--240},
      ISSN = {0065-1036,1730-6264},
   MRCLASS = {11N25 (05C38 11N35)},
  MRNUMBER = {1611201},
MRREVIEWER = {A.\ J.\ Hildebrand},
       DOI = {10.4064/aa-83-3-225-240},
       URL = {https://doi-org.dartmouth.idm.oclc.org/10.4064/aa-83-3-225-240},
}

@article{SBA15,
  title={Divisibility patterns of natural numbers on a complex network},
  author={Shekatkar, Snehal M and Bhagwat, Chandrasheel and Ambika, G},
  journal={Scientific Reports},
  volume={5},
  number={1},
  pages={1--13},
  year={2015},
  publisher={Nature Publishing Group},
  url={https://www.nature.com/articles/srep14280}
}

@article{SZ10,
  title={Natural number network and the prime number theorem},
  author={Shi, Dinghua and Zhang, Huijie},
  journal={Complex Systems and Complexity Science},
  volume={7},
  number={4},
  pages={1--9},
  year={2010}
}

\end{document}